\documentclass[12pt]{amsart}
\usepackage{graphicx}
\usepackage{amssymb}
\usepackage{xcolor}
\usepackage[shortlabels]{enumitem}

\usepackage[plainpages=false]{hyperref}

\numberwithin{equation}{section}

\usepackage{tikz}
\usepackage{mathdots}
\usepackage{yhmath}
\usepackage{cancel}
\usepackage{siunitx}
\usepackage{multirow}
\usepackage{tabularx}
\usepackage{extarrows}
\usepackage{booktabs}
\usetikzlibrary{fadings}
\usetikzlibrary{patterns}
\usetikzlibrary{shadows.blur}
\usetikzlibrary{shapes}
\allowdisplaybreaks

\newtheorem{thm}{Theorem}[section]

\newtheorem{lem}[thm]{Lemma}
\newtheorem{prop}[thm]{Proposition}
\newtheorem*{conj*}{Conjecture}
\theoremstyle{definition}

\theoremstyle{remark}

\begin{document}

\title{Flat level sets of Allen-Cahn equation in half-space}
\author{Wenkui Du}
\address{Department of Mathematics, Massachusetts Institute of Technology}
\email{\tt duwenkui@mit.edu}
\author{Ling Wang}
\address{School of Mathematical Sciences, Peking University}
\email{\tt lingwang@stu.pku.edu.cn}
\author{Yang Yang}
\address{Department of Mathematics, Johns Hopkins University}
\email{\tt yyang@jhu.edu}

%\thanks{1111}



\begin{abstract}
We prove a half-space Bernstein theorem for Allen-Cahn equation. More precisely, we show that every  solution $u$ of the Allen-Cahn equation in the half-space $\overline{\mathbb{R}^n_+}:=\{(x_1,x_2,\cdots,x_n)\in\mathbb{R}^n:\,x_1\geq 0\}$ with $|u|\leq 1$, boundary value
given by the restriction of a one-dimensional solution  on $\{x_1=0\}$ and monotone condition $\partial_{x_n}u>0$ as well as limiting condition $\lim_{x_n\to\pm\infty}u(x',x_n)=\pm 1$ must itself be one-dimensional,  
 and the parallel flat level sets and  $\{x_1=0\}$ intersect at the same fixed angle in $(0, \frac{\pi}{2}]$. 
\end{abstract}
\maketitle
% ----------------------------------------------------------------

%%%%%%%%%%%%%%%%%%%%%%%%%%%%%%%%%%%%%%%%%%%%%%%%%%%%%%%%%%%%%%%%%%%%%%%%%%%%%%%%%%%%%%%%%%%%%%%%%%%%%%%%%%%%
\section{Introduction}

In this paper,  we prove a half-space Bernstein theorem for  Allen-Cahn equation. This is related to a  half-space version De Giorgi's conjecture. We recall that the classical De Giorgi's conjecture was raised by De Giorgi in 1978 \cite{De}, which states as follows: 
\begin{conj*}[De Giorgi’s conjecture]
    If $u\in C^2(\mathbb{R}^n)$ is an entire solution of
    \begin{equation}\label{eq:Al-Ca}
        \Delta u=u^3-u,
    \end{equation}
    such that
    \begin{equation*}
        |u|\leq 1,\quad \partial_{x_n} u>0
    \end{equation*}
    in whole $\mathbb{R}^n$, then is it true that all the level sets of $u$ are hyperplanes, at least if $n \leq 8$? 
\end{conj*}
The conjecture is sometimes called ``the $\varepsilon$-version of the Bernstein problem for minimal graphs" because the level sets of $\varepsilon$-Allen-Cahn equation converges minimal hypersurface under some conditions (see \cite{Mo,HT,TW,Gu}). This relation and the Bernstein problem for minimal hypersurfaces explains why De Giorgi stated conjecture under dimension condtion $n \leq 8$.

De Giorgi’s Conjecture was proved true in dimension $n = 2$ by Ghoussoub and Gui \cite{GG} and for $n = 3$ by Ambrosio and Cabr\'e \cite{AC}. Savin \cite{Sa09} showed that for $4 \leq n \leq 8$, De Giorgi’s Conjecture holds under the additional natural limiting condition
\begin{equation}\label{eq:lim-assu}
    \lim_{x_n\to\pm\infty}u(x',x_n)=\pm 1
\end{equation}
holds pointwisely for every $x' = (x_1, \cdots , x_{n-1}) \in \mathbb{R}^{n-1}$. This condition implies that for every $\lambda \in (-1, 1)$, the level set $\{x \in \mathbb{R}^n : u(x) = \lambda\}$ of the function $u$ is entire graph with respect to the first $n - 1$ variables. Another proof of Savin's result provided by Wang can be found in \cite{Wa17}. On the other hand, for $n \geq 9$, del Pino, Kowalczyk and Wei \cite{dPKW} constructed monotone solutions which are not one-dimensional, so the dimension condition $n \leq 8$ in the De Giorgi's conjecture cannot be removed. It is worth noting that this counterexample also satisfies the limiting assumption \eqref{eq:lim-assu}.

Additionally, it is not hard to see that \eqref{eq:Al-Ca} is the Euler-Lagrange equation of energy functional 
\begin{equation}\label{eq:classical}
    J(u,\Omega):=\int_{\Omega} \frac{1}{2}|\nabla u|^2+\frac{1}{4}\left(1-u^2\right)^2\mathrm{d}x,\quad |u|\leq 1,
\end{equation}
where $\Omega$ is a $n$-dimensional domain in $\mathbb{R}^n$. One important rigidity result to highlight is the classification of solutions that are global minimizers of the associated energy functional \eqref{eq:classical} with $\Omega=\mathbb{R}^n$. Savin proved in the same paper \cite{Sa09} that global minimizers of \eqref{eq:classical} are one-dimensional for dimensions $n \leq 7$, while Liu, Wang, and Wei \cite{LWW} constructed counterexamples in dimensions $n \geq 8$. For more details and extensions to nonlinear equations' De Giorgi conjecture, we refer to \cite{CW, DS, Sa08, Sa09,Sa17, VSS} and the references therein. Recently, there have been some results related to the nonlocal De Giorgi conjecture (see \cite{CC14, DSV, Sa18, Sa19, SV09}) as well.

As the half-space Bernstein theorem for graphical minimal hypersurfaces considered in \cite{EW} and \cite{DMYZ}, it is natural to investigate whether the De Giorgi conjecture for Allen-Cahn equations is true in the
half-space. In \cite{HLSWW}, Hamel-Liu-Sicbaldi-Wang-Wei proved a half-space rigidity result under the assumption that the zero level set of the solution is contained in a half-space. Specifically, they showed that whenever $ n \leq 3 $ and $ u $ is a non-constant solution of \eqref{eq:Al-Ca} with the zero  level set contained in a half-space, $ u $ is one-dimensional. Farina and Valdinoci studied overdetermined problems for Allen-Cahn equations in \cite{FV} and subsequently obtained the half-space De Giorgi conjecture in dimensions $2$ and $3$. Note that in both results, the level sets of solutions are parallel to the boundary of half-space. In this paper, we obtain a different type of half-space rigidity result for Allen-Cahn equations, where the level sets of solutions of Allen-Cahn equations with double well potentials are allowed to have any fixed intersection angle in $(0, \frac{\pi}{2}]$ with the boundary of the half-space $\mathbb{R}^{n}_{+}$ under assuming 
$\partial_{x_n} u>0$ and the limiting condition in \eqref{eq:lim-assu}.

To state our result in this paper,  we consider general Allen-Cahn energy functional
\begin{equation}\label{eq:GL-ener}
    J(u,\Omega):=\int_{\Omega} \frac{1}{2}|\nabla u|^2+W(u)\,\mathrm{d}x,\quad |u|\leq 1,
\end{equation}
where $W$ is a double-well potential achieving minimum at $1$ and $-1$  and satisfying
\begin{equation}\label{W1}
        W\in C^2([-1,1]),\quad W(-1)=W(1)=0,\quad W>0\text{ on }(-1,1),
        \end{equation}
\begin{equation}\label{W2}
        W'(-1)=W'(1)=0,\quad W''(-1)>0,\quad W''(1)>0.
        \end{equation}
% which gives us the second term in \eqref{eq:classical}.
The Euler-Lagrange equation of Allen-Cahn energy functional \eqref{eq:GL-ener} is
\begin{equation}\label{eq:Al-Ca-W}
    \Delta u=W'(u),
\end{equation}
and in \eqref{eq:Al-Ca} $W(s)=\frac{1}{4}(1-s^2)^2$ is the classical double-well potential.

If we define
\begin{equation}\label{g}
H_0(s):=\int_0^s\frac{1}{\sqrt{2W(\xi)}}\mathrm{~d}\xi,\quad\text{and}\quad g_0(t):=H_0^{-1}(t),
\end{equation}
then we find that 
$$g_0''(t)=W'(g_0),$$
and $g_0$ is the unique strictly increasing solution of \eqref{eq:Al-Ca-W} and is called as the one-dimensional solution of \eqref{eq:Al-Ca-W}.
% First, recall that in the simplest case $n = 1$, equation \eqref{eq:Al-Ca} reduces to an ODE and has a unique strictly monotone solution
% \begin{equation}\label{eq:1D-solu}
%      \mathbb{H}(x)=\operatorname{tanh}\left(\frac{x}{\sqrt{2}}\right).
% \end{equation}
On the other hand, a solution of \eqref{eq:Al-Ca-W} in $\mathbb{R}^n$ with parallel flat level sets must be of the form
    $u(x)= g_0(a\cdot x+c)$,
where $c\in\mathbb{R}$, and $a$ is any unit vector in $\mathbb{R}^n$. 

Now, we recall that we define $\mathbb{R}^{n}_{+}:=\{(x_1,x_2,\cdots,x_n)\in\mathbb{R}^{n}:\,x_1>0\}$, $\partial\mathbb{R}^{n}_{+}=\{(x_1,x_2,\cdots,x_n)\in\mathbb{R}^{n}: x_1=0\}$ and we have the following main result about half-space Bernstein theorem for Allen-Cahn equations.

%%%%% The statement of main theorem needs to be confirmed

\begin{thm}\label{thm:main}
    If $u\in C^2\left(\mathbb{R}^{n}_{+}\right)\cap C\left(\overline{\mathbb{R}^{n}_{+}}\right)$ is a solution to 
    \begin{equation}\label{eq:half-Al-Ca}
        \left\{
        \begin{aligned}
            &\Delta u=W'(u) && \text{in}&& {\mathbb{R}^{n}_{+}},\\
            &u(x)=g_0(a\cdot x) &&\text{on} && \partial\mathbb{R}^{n}_{+}=\{x_1=0\},\\
            &|u|\leq 1, \:\partial_{x_n}u>0&&\text{in}&& \overline{\mathbb{R}^{n}_{+}},
        \end{aligned}
        \right.
    \end{equation}
    where double well potential $W$ satisfies \eqref{W1}-\eqref{W2}, $a:=(a_1,a_2,\cdots,a_n)$ is any unit vector in $\mathbb{R}^n$ with $a_n>0$. Assuming further that $u$ satisfies the limiting condition\footnote{By the example constructed by Andersson in \cite{An},  the limiting condition assumption \eqref{eq:half-lim-assu}  in Theorem \ref{thm:main} is necessary when $a_1\not =0$. Indeed, Andersson's non-one-dimensional counterexample to half-space De Giorgi’s conjecture for \eqref{eq:half-Al-Ca} with $W'(u)=u^3-u$ in $\mathbb{R}^{2}_{+}$ has the boundary condition $g_{0}(a_1x_1+a_2x_2)$ with $g_0(x)=\textrm{tanh}(\frac{x}{\sqrt{2}})$ and $(a_1, a_2)=(\frac{\sqrt{2}}{2}, \frac{\sqrt{2}}{2})$  and his argument for showing existence of counter example does not hold when $a_1=0$. It would be interesting to know if limiting condition \eqref{eq:half-lim-assu} Theorem \ref{thm:main} can be removed or not when $a_1=0$.}
    \begin{equation}\label{eq:half-lim-assu}
        \lim_{x_n\to\pm\infty}u(x',x_n)=\pm 1
    \end{equation}
    pointwise for any $x' \in \mathbb{R}^{n-1}_{+}$, then we have that $u$ must be one of the following one-dimensional solution:
    $$u(x)=g_0(\pm a_1x_1+a_2x_2+\cdots+a_nx_n).$$
    In particular, if the above $a=(0,a_2,\cdots,a_n)$ is any unit vector in $\{0\}\times\mathbb{R}^{n-1}$ with $a_n>0$, i.e. boundary condition function $ g_0(a_2x_2+\cdots+a_nx_n)$ is a one-dimensional solution to the equation \eqref{eq:Al-Ca-W} on $\mathbb{R}^{n-1}$, then $u(x)= g_0(a_2x_2+\cdots+a_nx_n)$ is the one-dimensional solution whose level sets are orthogonal to the boundary of the half-space.
\end{thm}

It is worth noting that, unlike the De Giorgi conjecture proved by Savin \cite{Sa09}, which assumes dimension $n \leq 8$, our main result Theorem \ref{thm:main} holds in all dimensions, just as the affirmative answer to  the half-space Bernstein problem for graphical minimal hypersurfaces with linear boundary conditions also holds in all dimensions (see \cite{EW} or \cite{DMYZ}). Our result also allows the level sets of entire solutions to have any fixed intersection angle in $(0, \frac{\pi}{2}]$ with the boundary of the half-space, in contrast to the level sets of entire solutions in lower dimensions $n=2, 3$ being parallel to the boundary of the half-space in  \cite{HLSWW} and \cite{FV}.

The key idea for proving Theorem \ref{thm:main} is to extend the blowdown method  in the proof of the half-space Bernstein theorem for anisotropic minimal graphs (see \cite{DMYZ}), to Savin's framework \cite{Sa09} for Allen-Cahn equations. Using the linear boundary condition, limiting condition \eqref{eq:half-lim-assu} and $\partial_{x_n}u>0$, we show that solution is minimizer of Allen-Cahn energy. While the level sets of Allen-Cahn energy minimizing solution to \eqref{eq:half-Al-Ca} generally do not satisfy any equation, we can still show that in some weak sense there are equations that can be satisfied. By passing to the limit of rescaled solutions to \eqref{eq:half-Al-Ca}, we obtain that the limiting level set satisfies the graphical minimal hypersurface equation in the viscosity sense. Therefore, we can apply results of half-space Bernstein problem for graphical minimal hypersurfaces with linear boundary conditions from \cite{EW} or \cite{DMYZ} to conclude that this level set must be a graphical half hyperplane. Then, similar to the blowdown method in \cite{DMYZ}, we bound the level sets of solutions to \eqref{eq:half-Al-Ca} by two critical hyperplanes with the same boundary. Using a barrier argument similar to the Hopf-type lemma proved in \cite[Lemma 3.1]{DMYZ} and sliding method, we show that these two hyperplanes must coincide with the hyperplane obtained in the limit. Consequently, the level sets are flat. Since this procedure holds for all level sets, we conclude that the solutions to \eqref{eq:half-Al-Ca} with limiting condition \eqref{eq:half-lim-assu} are one-dimensional.

% In this context, an ABP-type measure estimate and a boundary Harnack inequality are crucial. More precisely, we hope to use blowdown argument first to prove global minimizer of \eqref{eq:GL-ener} with the boundary condition is one-dimensional.  
% By passing to the limit of rescalling solutions to \eqref{eq:half-Al-Ca}, we can apply results from \cite{EW} or \cite{DMYZ} to show the existence of points projected into the interior of two fat level sets in the $(n+1)$-dimensional direction, rather than at the boundary intersection of the two surfaces. We then apply Savin's interior ABP regarding projection measures and the inequality characterizing the points in Savin's ABP measure lemma to show that either the normal vectors of the two half-space surfaces are the same or that a contradiction arises from minimizing due to the overlapping sheets in the context of Allen-Cahn energy computations. Consequently, the level sets are flat.

The structure of this paper is organized as follows. In Section \ref{sec:lems}, we present some lemmas that will be used in the proof of the main theorem. In Section \ref{sec:pfs}, we give the proof of Theorem \ref{thm:main}.

\vskip 10pt
\noindent\textbf{Acknowledgments.}
W. Du acknowledges the support of
the MIT. L. Wang acknowledges his PhD supervisor, Professor Bin Zhou, for his constant encouragement and support. Y. Yang acknowledges the support of
the Johns Hopkins University Provost’s Postdoctoral Fellowship Program. 

%%%%%%%%%%%%%%%%%%%%%%%%%%%%%%%%%%%%%%%%%%%%%%%%%%%%%%%%%%%%%%%%%%%%%%%%%%%%%%%%%%%%%%%%%%%%%%%%%%%%%%%%%%%%

\section{Preliminary}\label{sec:lems}

%%%%% This needs to be modified

In this section, we present some lemmas that will be used in the proof of the main theorem. The first lemma is a minor modification of the Modica theorem in viscosity sense proved in \cite[Proposition 5.1]{Sa09}, which says that the $0$ level set of a local minimizer of \eqref{eq:GL-ener} satisfies the mean curvature equation in some weak viscosity sense, where the size of the neighborhood around the touching point must be specified.

\begin{prop}[Modica theorem-viscosity sense]\label{prop:Modica}
    Let $u$ be a minimizer of \eqref{eq:GL-ener} and assume that $u(0)=0$. Consider the graph of a $C^2$ function
    $$\Gamma=\left\{(x',x_n):\,x_n=w(x'),\, w(0')=0,\,Dw(0')=0\right\}$$
    that satisfies
    \begin{equation}\label{eq:posi-curv}
        \Delta w(0')>\delta_0\|D^2w(0')\|,\quad \|D^2w(0')\|<\delta_0^{-1},
    \end{equation}
    at the origin $0'\in \mathbb{R}^{n-1}$ for some $\delta_0>0$ small. Let $u_\varepsilon(x):=u(\varepsilon^{-1}x)$ be minimizer of 
    \begin{equation*}
        J(u_{\varepsilon},\Omega):=\int_{\Omega} \frac{\varepsilon }{2}|\nabla u|^2+\frac{W(u)}{\varepsilon}\,\mathrm{d}x,\quad |u_{\varepsilon}|\leq 1,
    \end{equation*}

    There exists $\sigma_0(\delta_0)>0$ small, such that if $\varepsilon\leq\sigma_0(\delta_0)$ then $\Gamma$ cannot touch from below $\{u_{\varepsilon}=0\}$ at $0$ in a $\delta_0\varepsilon^{\frac{1}{2}}(\Delta w(0'))^{-\frac{1}{2}}$ neighborhood; more explicitly,
    $$\{u_{\varepsilon}=0\}\cap\left\{x_n<w\right\}\cap\left\{|x|<\delta_0\varepsilon^{\frac{1}{2}}(\Delta w(0'))^{-\frac{1}{2}}\right\}\neq \emptyset.$$
\end{prop}
\begin{proof}
    In \cite[Proposition 5.1]{Sa09}, Savin proved this statement with $w(x')=\frac{1}{2}x'^{T}Mx'$. Following a similar argument based on the equivalent definition of viscosity solutions, we can prove this statement as well. The details are given below. 

    Let
    $$P(x')=\frac{1}{2}x'^{T}D^2w(0') x'-\frac{t}{2}|x'|^2.$$
    Then by \eqref{eq:posi-curv}, we can choose sufficient small $t>0$ such that
    $$\Delta P>\delta_0\|D^2w(0')-tI\|,\quad \|D^2w(0')-tI\|<\delta_0^{-1}.$$
    Hence by \cite[Proposition 5.1]{Sa09}, we know that there exists $\sigma_0(\delta_0)>0$ small, such that if $\varepsilon\leq\sigma_0(\delta_0)$ then
    $$\{u_{\varepsilon}=0\}\cap\left\{x_n<P\right\}\cap\left\{|x|<\delta_0\varepsilon^{\frac{1}{2}}(\Delta P)^{-\frac{1}{2}}\right\}\neq \emptyset.$$
    Since $\{x_n<P\}\subset\{x_n<w\}$ and a little perturbation of $t$, we have
    $$\{u_{\varepsilon}=0\}\cap\left\{x_n<w\right\}\cap\left\{|x|<\delta_0\varepsilon^{\frac{1}{2}}(\Delta w(0'))^{-\frac{1}{2}}\right\}\neq \emptyset,$$
    which is $\Gamma$ cannot touch from below $\{u_{\varepsilon}=0\}$ at $0$ in a $\delta_0\varepsilon^{\frac{1}{2}}(\Delta w(0'))^{-\frac{1}{2}}$ neighborhood.
\end{proof}
Proposition \ref{prop:Modica} shows that the level set $\{u_{\varepsilon}=0\}$ is, in some sense, a viscosity supersolution of the minimal surface equation. Of course, by similar argument in Proposition \ref{prop:Modica}, the level set $\{u_{\varepsilon}=0\}$ is also a viscosity subsolution of the minimal surface equation.

As a corollary of the above lemma, we conclude that if $ \{u_{\varepsilon} = 0\} $ converges uniformly to a surface, then this surface satisfies the zero mean curvature equation in the viscosity sense. We state this result precisely below.
\begin{lem}[{\cite[Theorem 2.3]{SV}}]\label{lem:v-solu}
    Let $u$ be a minimizer of \eqref{eq:GL-ener} and $u_{\varepsilon}:=u(\varepsilon^{-1}x)$. If $u_{\varepsilon}$ converges in $L^1_{loc}$ to $\chi_{E}-\chi_{\mathbb{R}^n\backslash E}$ and $\{u_{\varepsilon} = 0\}$ converges locally uniformly to $S:=\partial E$, then $S$ satisfies the zero mean curvature equation in the viscosity sense.
\end{lem}
\begin{proof}
    This result was proved in \cite[Theorem 2.3]{SV} for general $ p $-Laplace phase transitions; here, we apply it for the case $ p = 2 $.
\end{proof}

The final lemma in this section addresses interior gradient estimates for viscosity solutions to the minimal surface equation. This result will later be used to show that viscosity solutions to the minimal surface equation are, in some sense, classical.

\begin{lem}[\textit{a priori} estimate of the gradient]\label{lem:int-gra-est}
    Let $\gamma$ be a viscosity solution to the minimal surface equation, 
    $$\sum_{i,j=1}^{n}\left(\delta_{ij}-\frac{D_i\gamma D_j\gamma}{1+|D\gamma|^2}\right)D_{ij}\gamma=0$$
    in $B_R(x_0)$. Then there exists a constant $C>0$ such that
    \begin{equation}\label{eq:int-grad-est}
        \sup_{B_{R/2}(x_0)}|D\gamma|\leq\exp\left[C\left(1+\frac{\sup_{B_R(x_0)}\gamma-\gamma(x_0)}{R}\right)\right].
    \end{equation}
\end{lem}
\begin{proof}
    This estimate was first derived under the assumption $\gamma \in C^2$ by Bombieri, De Giorgi, and Miranda \cite{BDM}. Since then, several simpler and more modern proofs have been provided by various authors. For the viscosity solution case, Wang stated in \cite[Theorem 1.1]{Wa} that the estimate can be achieved via an approximation argument.
\end{proof}

%%%%%%%%%%%%%%%%%%%%%%%%%%%%%%%%%%%%%%%%%%%%%%%%%%%%%%%%%%%%%%%%%%%%%%%%%%%%%%%%%%%%%%%%%%%%%%%%%%%%%%%%%%%%
\section{Proof of the main theorem}\label{sec:pfs}

In this section, we give the proof of Theorem \ref{thm:main}. Before doing so, we first show that a function $u$ satisfying \eqref{eq:half-Al-Ca} and \eqref{eq:half-lim-assu} is a global minimizer in $\mathbb{R}^{n}_{+}$. The idea is similar to the proof of \cite[Theorem 2.4]{Sa09}, with some details adapted from the proof of \cite[Lemma 9.1]{VSS}.

%%%% Need to discuss
\begin{lem}\label{lem:global-min}
    If $u\in C^2\left(\mathbb{R}^{n}_{+}\right)\cap C\left(\overline{\mathbb{R}^{n}_{+}}\right)$ satisfies \eqref{eq:half-Al-Ca} and \eqref{eq:half-lim-assu}, then $u$ is a global minimizer of \eqref{eq:GL-ener} in $\mathbb{R}^{n}_{+}$.
\end{lem}
\begin{proof}
    Since $\partial_{x_n} u > 0$, i.e. $u$ is strictly increasing, we actually know $|u| < 1$ by \cite[footnote in page 2]{VSS}. Let ${\bf B} \subset {\mathbb{R}^{n}}$ be a closed $n$-dimensional ball, and let $v$ be a minimizer of $J(v, {\bf B}\cap\overline{\mathbb{R}^n_+})$ such that $v = u$ on $\partial({\bf B}\cap\overline{\mathbb{R}^n_+})$. Our goal is to show that $u = v$ in ${\bf B}\cap\overline{\mathbb{R}^n_+}$. If ${\bf B} \cap \partial \mathbb{R}^{n}_{+} = \emptyset$, this follows from \cite[Theorem 2.4]{Sa09} or \cite[Lemma 9.1]{VSS}. Thus, it suffices to consider the case where ${\bf B} \cap \partial \mathbb{R}^{n}_{+} \neq \emptyset$. We will prove this claim by contradiction. Assume that there exists a point $x^{*}\in {\bf B}\cap\overline{\mathbb{R}^n_+}$ such that
    \begin{equation}\label{eq:v>u}
        v(x^{*})>u(x^{*}).
    \end{equation}
    Hence by the boundary assumption $u=v$, we know that $x^{*}$ is in the interior of ${\bf B}\cap\overline{\mathbb{R}^n_+}$.
    
    Note that the boundary condition in \eqref{eq:half-Al-Ca}, which is a one-dimensional solution restricts to $\partial\mathbb{R}^n_+$. By the global gradient estimate for semilinear elliptic equations (with the above boundary condition), we know that for all $ x \in \overline{\mathbb{R}^n_+} $, there exists a constant $ C > 0 $ (does not depend on $ x $), such that $ |\nabla u| \leq C $ in $ {\bf B}_{1}(x) \cap \overline{\mathbb{R}^n_+} $ (see, for example, \cite[Proposition 2.20]{HL}).  Combining $\displaystyle\lim_{x_n\to+\infty}u(x',x_n)=1$, we deduce that
    \begin{equation}\label{ineq}
        u(x+te_n)\geq v(x)
    \end{equation}
    for any $x\in  {\bf B}\cap\overline{\mathbb{R}^n_+}$ provided that $t$ is large enough. Indeed, if this is not true, we have $u(x_t+te_n)<v(x_t)$ for some $x_t \in {\bf B}\cap\overline{\mathbb{R}^n_+}$ and a diverging sequence of $t$. Note that there is $\alpha > 0$ so that $v \leq 1 - \alpha$. Then, up to subsequence, we may assume that $x_t$ converges to $x_{\infty}\in {\bf B}\cap\overline{\mathbb{R}^n_+}$, then
    \begin{align*}
        1&=\lim_{t\to+\infty}u(x_{\infty}+te_n)\\
        &=\lim_{t\to+\infty}\left[u(x_{\infty}+te_n)-u(x_t+te_n)+u(x_t+te_n)\right]\\
        &\leq \lim_{t\to+\infty}\left[u(x_t+te_n)+C|x_{\infty}-x_t|\right]\\
        &=\lim_{t\to+\infty}u(x_t+te_n)\\
        &\leq \lim_{t\to+\infty}v(x_t)\\
        &\leq 1-\alpha,
    \end{align*}
    which make a contradiction. Thanks to the inequality (\ref{ineq}), we thus slide $u(\cdot + te_n)$ along the $e_n$-direction until we touch $v$ from above for the first time. Say this happen at $\bar{x}\in {\bf B}\cap\overline{\mathbb{R}^n_+}$ for $t=\bar{t}$. Then by \eqref{eq:v>u}, we have
    $$u(x^{*}+\bar{t}e_n)\geq v(x^{*})>u(x^{*}),$$
    thence, since $u$ is strictly increasing in the $e_n$-direction (in ${\bf B} \cap  \mathbb{R}^{n}_{+}$, we use $\partial_{x_n} u>0$ and in ${\bf B} \cap \partial \mathbb{R}^{n}_{+}$ we use one-dimensional strictly increase boundary condition $u(x)=g_0(a\cdot x)$ with $a_n>0$), we know that $\bar{t}>0$.
    Since now $\partial_{x_n}u> 0$ we have that $\nabla u(\cdot+\bar{t}e_n)\neq 0$. Therefore, it follows that the assumptions of the Strong Comparison Principle for quasilinear degenerate elliptic equations in \cite[Theorem 1.4]{Da} applies to $u(\cdot+\bar{t}e_n)$ and $v(\cdot)$ and so this touching point must occur on $\partial ({\bf B}\cap\overline{\mathbb{R}^n_+})$, that is $\bar{x}\in\partial ({\bf B}\cap\overline{\mathbb{R}^n_+})$. Since $u = v$ on $\partial ({\bf B}\cap\overline{\mathbb{R}^n_+})$, it follows that $v(\bar{x}) = u(\bar{x})$. Note that as above, by $\partial_{x_n}u>0$ and one-dimensional strictly increase boundary condition $u$ is strictly increasing in the $e_n$-direction. Hence $$u(\bar{x})=v(\bar{x})=u(\bar{x}+\bar{t}e_n)>u(\bar{x}),$$
    contradiction, which means that \eqref{eq:v>u} cannot hold. Hence $u\geq v$. Analogously, one can see that $u \leq v$, thence $u = v$.
\end{proof}

Since we have established the minimality of $ u $, we can combine the Modica theorem with the half-space Bernstein theorem for minimal surfaces to conclude that the level sets of rescaled solutions to \eqref{eq:Al-Ca} locally uniformly converge to 
hyperplanes in all dimensions.
\begin{lem}\label{lem:converges}
    Let $u\in C^2\left(\mathbb{R}^{n}_{+}\right)\cap C\left(\overline{\mathbb{R}^{n}_{+}}\right)$ be a solution to \eqref{eq:half-Al-Ca} with \eqref{eq:half-lim-assu}, and $u_{\varepsilon}(x):=u(\varepsilon^{-1}x)$. Then there exists a subsequence $\{\varepsilon_k\}\to 0$ such that the level sets
    $$\left\{x\in\overline{\mathbb{R}^{n}_{+}}:u_{\varepsilon_k}(x)=0\right\}=\varepsilon_k\left\{x\in\overline{\mathbb{R}^{n}_{+}}:u(x)=0\right\}$$
    converge uniformly on compact sets to a graphical half-hyperplane.
\end{lem}
\begin{proof}
By the minimality of $ u $ (see Lemma \ref{lem:global-min}) and the Modica theorem, we know that $ u_\varepsilon $ converges in $ L^{1}_{\text{loc}} $ (and thus a.e. converges), up to a subsequence, to $ \chi_E - \chi_{\overline{\mathbb{R}^n_+} \backslash E} $ for a suitable set $ E \subset \overline{\mathbb{R}^n_+} $ with minimal perimeter. Following the arguments in \cite[Pages 80-81]{VSS}, we can conclude that there exists a measurable function $ \gamma_{*}: \overline{\mathbb{R}^{n-1}_+}\to [-\infty,+\infty] $, where $$\mathbb{R}^{n-1}_+:=\{x'=(x_1,x_2,\cdots,x_{n-1})\in\mathbb{R}^{n-1}: \,x_1>0\},$$ such that $ \overline{\mathbb{R}^n_+} \backslash E = \left\{ x \in \overline{\mathbb{R}^n_+} : x_n < \gamma_{*}(x') \right\} $, and that $ \left.\gamma_{*}\right|_{\partial\mathbb{R}^{n-1}_+} $ is linear. By Lemma \ref{lem:v-solu}, we know that $\gamma_{*}$ satisfies the minimal surface equation in the viscosity sense  with linear boundary value conditions. Next we show that $\gamma_{*}$ is smooth in the interior. 

For any $x_0\in\mathbb{R}^{n-1}_+$, we can choose a small enough $r>0$ such that $B_r(x_0)\subset\subset\mathbb{R}^{n-1}_+$. Then by Lemma \ref{lem:int-gra-est}, we know that $|D\gamma_{*}|$ is bounded in $B_{r/2}(x_0)$. Then by approximation argument as in \cite{Tr}, the function $w=D_s\gamma_{*}$, $(s=1,2,\cdots,n-1)$ satisfies the equation
\begin{equation}\label{eq:w}
    D_i\left(a_{ij}(x)D_jw\right)=0
\end{equation}
with
$$a_{ij}(x)=\frac{\delta_{ij}(1+|D\gamma_{*}|^2)-D_i\gamma_{*}D_j\gamma_{*}}{(1+|D\gamma_{*}|^2)^{3/2}}\in L^\infty(B_{r/2}(x_0)).$$
Since we have proved $|D\gamma_{*}|$ is bounded in $B_{r/2}(x_0)$, we have for every $\xi\in\mathbb{R}^{n-1}$
$$\nu|\xi|^2\leq a_{ij}(x)\xi_i\xi_j$$
with some constant $\nu>0$. Thus, equation \eqref{eq:w} is uniformly elliptic. Then by De Giorgi-Nash-Moser theory \cite[Theorem 8.24]{GT}, we know that $w\in C^\alpha(B_{r/4}(x_0))$ for some $\alpha>0$, which is $\gamma_{*}\in C^{1,\alpha}(B_{r/4}(x_0))$. Hence, by the classical Schauder theory \cite[Theorem 6.2]{GT}, we can conclude that $\gamma_{*}$ is smooth at $x_0$.

Therefore, by the half-space Bernstein theorem for minimal graphs in \cite{EW} or \cite{DMYZ}, we can conclude that $\partial E=\left\{ x \in \overline{\mathbb{R}^n_+} : x_n = \gamma_{*}(x') \right\}$ is a graphical half-hyperplane in $\overline{\mathbb{R}^n_+}$. By the boundary condition, we obtain that $\partial E\cap\{x_1=0\}=\left\{u=0\right\}\cap\left\{x_1=0\right\}$. Then thanks to the density estimates of Caffarelli and Cordoba \cite[Page 11]{CC}, we know $\{u_{\varepsilon}=0\}$ $L^\infty_{\text{loc}}$-converges to $\partial E$.
\end{proof}

%%%%%%
Next, to apply a similar blowdown method as in \cite{DMYZ}, we need to extend certain notions and lemmas to the level sets of solutions to \eqref{eq:half-Al-Ca}. Given the assumption that $\partial_{x_n} u > 0$, we know that the level sets are graphs along the $x_n$-direction. By the boundary value condition and the definition of $g_0$ in (\ref{g}), we have
\begin{align*}
    \Gamma:&=\left\{u(x)=0\right\}\cap\left\{x_1=0\right\}=\left\{x\in\overline{\mathbb{R}^{n}_{+}}:g_0(a\cdot x)=0,\,x_1=0\right\}\\&=\left\{x\in\overline{\mathbb{R}^{n}_{+}}:a_2x_2+\cdots a_nx_n=0,\,x_1=0\right\}.
\end{align*}
Hence, we define
\begin{equation}\label{eq:def-A}
    \begin{aligned}
        A_{+}:&=\sup\left\{A:\{u=0\}\subset\{x\in\overline{\mathbb{R}^{n}_{+}}:Ax_1+a_2x_2+\cdots+a_nx_n\leq 0\}\right\},\\
        A_{-}:&=\inf\left\{A:\{u=0\}\subset\{x\in\overline{\mathbb{R}^{n}_{+}}:Ax_1+a_2x_2+\cdots+a_nx_n\geq 0\}\right\},
    \end{aligned}
\end{equation}
where $A_{+}\in\mathbb{R}\cup\{-\infty\}$ and $A_{-}\in\mathbb{R}\cup\{+\infty\}$. By definition, it is clear that $A_{+}\leq A_{-}$. To show that $\{u=0\}$ is flat, it suffices to prove that $A_{+}=A_{-}$.

Let $ H_{\pm} $ denote the critical hyperplanes $ \{ A_{\pm} x_1 + a_2 x_2 + \cdots + a_n x_n = 0 \} $. Note that our definition of $ A_{\pm} $ has the opposite sign compared to that in \cite{DMYZ}, but $ H_{\pm} $ are the same. Thus, when $ A_+ = -\infty $, we interpret $ H_+ $ as the closed half-space in $ \{ x_1 = 0 \} $ lying above $ \Gamma $, and we understand $ H_- $ similarly when $ A_- = +\infty $. See Figure \ref{fig:H_pm}.
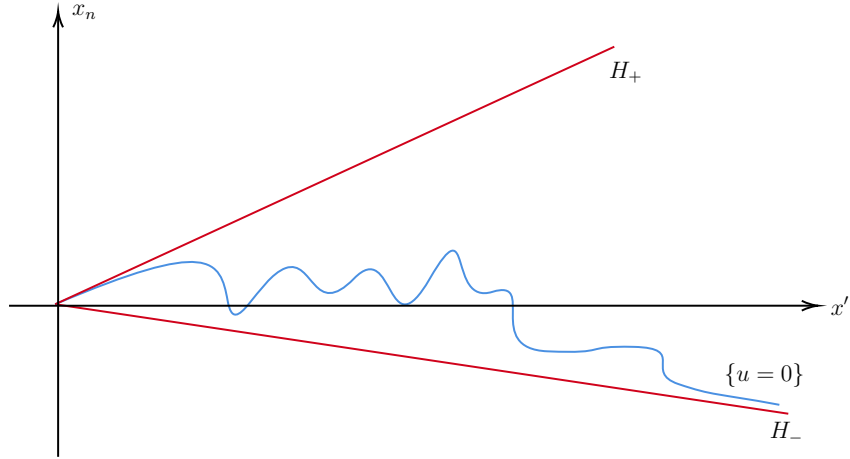
\begin{figure}[htbp]
    \centering
%%%%tikz
\tikzset{every picture/.style={line width=0.75pt}} %set default line width to 0.75pt        

\begin{tikzpicture}[x=0.75pt,y=0.75pt,yscale=-0.75,xscale=0.75]
%uncomment if require: \path (0,330); %set diagram left start at 0, and has height of 330

%Straight Lines [id:da22882799423139] 
\draw [color={rgb, 255:red, 208; green, 2; blue, 27 }  ,draw opacity=1 ][line width=0.75]    (71,205.83) -- (564,279.5) ;
%Curve Lines [id:da8714992834799942] 
\draw [color={rgb, 255:red, 74; green, 144; blue, 226 }  ,draw opacity=1 ]   (71,205.83) .. controls (175,161.5) and (182,175.5) .. (187,203.5) .. controls (192,231.5) and (207.33,188.17) .. (226,181.5) .. controls (244.67,174.83) and (244.67,216.5) .. (272,188.5) .. controls (299.33,160.5) and (293.67,239.17) .. (323,188.5) .. controls (352.33,137.83) and (334.33,206.83) .. (365,197.5) .. controls (395.67,188.17) and (358.67,235.67) .. (402,237.5) .. controls (445.33,239.33) and (423.67,233.67) .. (460,234.5) .. controls (496.33,235.33) and (466.14,251.38) .. (490,259.5) .. controls (513.86,267.62) and (515,265.5) .. (558,273.5) ;
%Straight Lines [id:da05106004207835868] 
\draw [color={rgb, 255:red, 208; green, 2; blue, 27 }  ,draw opacity=1 ][line width=0.75]    (71,205.83) -- (447,32.5) ;
%Straight Lines [id:da02342979196560946] 
\draw    (40,206.83) -- (582,206.8) ;
\draw [shift={(584,206.8)}, rotate = 180] [color={rgb, 255:red, 0; green, 0; blue, 0 }  ][line width=0.75]    (10.93,-3.29) .. controls (6.95,-1.4) and (3.31,-0.3) .. (0,0) .. controls (3.31,0.3) and (6.95,1.4) .. (10.93,3.29)   ;
%Straight Lines [id:da10002129667904502] 
\draw    (73,308.5) -- (73,10.5) ;
\draw [shift={(73,8.5)}, rotate = 90] [color={rgb, 255:red, 0; green, 0; blue, 0 }  ][line width=0.75]    (10.93,-3.29) .. controls (6.95,-1.4) and (3.31,-0.3) .. (0,0) .. controls (3.31,0.3) and (6.95,1.4) .. (10.93,3.29)   ;

% Text Node
\draw (519.67,244) node [anchor=north west][inner sep=0.75pt]  [xscale=0.75,yscale=0.75] [align=left] {$\displaystyle \{u=0\}$};
% Text Node
\draw (550,283.5) node [anchor=north west][inner sep=0.75pt]  [xscale=0.75,yscale=0.75] [align=left] {$\displaystyle H_{-}$};
% Text Node
\draw (442,41) node [anchor=north west][inner sep=0.75pt]  [xscale=0.75,yscale=0.75] [align=left] {$\displaystyle H_{+}$};
% Text Node
\draw (81,3) node [anchor=north west][inner sep=0.75pt]  [xscale=0.75,yscale=0.75] [align=left] {$\displaystyle x_{n}$};
% Text Node
\draw (591,199) node [anchor=north west][inner sep=0.75pt]  [xscale=0.75,yscale=0.75] [align=left] {$\displaystyle x'$};
\end{tikzpicture}
    \caption{critical hyperplanes $H_{\pm}$}
    \label{fig:H_pm}
\end{figure}

Denote $u_\varepsilon(x) = u(\varepsilon^{-1}x)$, then we have
\begin{equation}\label{eq:u-level}
    \left\{x\in\overline{\mathbb{R}^{n}_{+}}:u_\varepsilon(x)=0\right\}=\varepsilon\left\{x\in\overline{\mathbb{R}^{n}_{+}}:u(x)=0\right\}.
\end{equation}

\begin{lem}\label{lem:A_+=A_-}
    Let $u\in C^2\left(\mathbb{R}^{n}_{+}\right)\cap C\left(\overline{\mathbb{R}^{n}_{+}}\right)$ be a solution to \eqref{eq:half-Al-Ca} with \eqref{eq:half-lim-assu}, then we have $A_+=A_-$.
\end{lem}
\begin{proof}
    By the definition of $A_{\pm}$ in \eqref{eq:def-A}, we know that
    $$\{x\in\overline{\mathbb{R}^{n}_{+}}:u(x)=0\}\subset\{x\in\overline{\mathbb{R}^{n}_{+}}:A_+x_1+a_2x_2+\cdots+a_nx_n\leq 0\}$$
    and
    $$\{x\in\overline{\mathbb{R}^{n}_{+}}:u(x)=0\}\subset\{x\in\overline{\mathbb{R}^{n}_{+}}:A_-x_1+a_2x_2+\cdots+a_nx_n\geq 0\},$$
    i.e. $\{u=0\}$ lies between $H_+$ and $H_-$. Since $H_{\pm}$ are Lipschitz scaling invariant, combining \eqref{eq:u-level}, we get
    $$\{x\in\overline{\mathbb{R}^{n}_{+}}:u_\varepsilon(x)=0\}\subset\{x\in\overline{\mathbb{R}^{n}_{+}}:A_+x_1+a_2x_2+\cdots+a_nx_n\leq 0\}$$
    and
    $$\{x\in\overline{\mathbb{R}^{n}_{+}}:u_\varepsilon(x)=0\}\subset\{x\in\overline{\mathbb{R}^{n}_{+}}:A_-x_1+a_2x_2+\cdots+a_nx_n\geq 0\},$$
    which means that $\{u_\varepsilon=0\}$ still lies between $H_{\pm}$. 

    Since $\partial_{x_n}u>0$, we know that $u$ is a graph in the $x_n$-direction. Hence, we denote 
    $$\{u=0\}=\{x \in \overline{\mathbb{R}^n_+} :\, x_n=\gamma(x')\},$$ 
    where $\gamma:\,\overline{\mathbb{R}^n_+}\to\mathbb{R}$ is a $C^2$-function satisfying
    $$u(x',\gamma(x'))=0,\quad\text{and}\quad \gamma(0,x_2,\cdots,x_{n-1})=-a_n^{-1} (a_2 x_2 + \cdots + a_{n-1} x_{n-1}).$$
    Denote $\gamma_{\varepsilon}(x'):=\varepsilon\gamma(\varepsilon^{-1}x')$. Then we have
    $$\{u_\varepsilon=0\}=\{x\in \overline{\mathbb{R}^n_+}: x_n=\gamma_{\varepsilon}(x')\}.$$
    By Lemma \ref{lem:converges}, there is sequence $\{\varepsilon_k\}$ such that $\{u_{\varepsilon_k}=0\}$ locally uniformly converges to a graphical half-hyperplane, denoted by $\partial E$, and
    $$\partial E=\{ x \in \overline{\mathbb{R}^n_+} :\, x_n = \gamma_{*}(x') \}\quad\text{with}\quad\partial E\cap\{x_1=0\}=\Gamma,$$
    where $\gamma_{*}(x'):=\lim_{k\to\infty}\gamma_{\varepsilon_k}(x')$ is a linear function in $\mathbb{R}^{n-1}_+$.
    Thus, we know that $\partial E$ still lies between $H_{\pm}$. Next, we show that the hyperplanes $H_{\pm}$ must coincide with the hyperplane $\partial E$, which implies that $A_+=A_-$. 
    
   Suppose towards a contradiction $A_{+}<A_{-}$ We first consider the case where $A_{-}<+\infty$. Suppose that $ H_- $ lies below $ \partial E $, i.e., the plane $ \partial E $ and $ H_- $ form a positive angle on the boundary $ \Gamma $.  Denote $ e_1' := (1, 0, \dots, 0) \in \mathbb{R}^{n-1}_+ $. Then, using a similar subsolution barrier function as in \cite[Lemma 3.1]{DMYZ}, we define
    $$
    w = -a_n^{-1} (A_- x_1 + a_2 x_2 + \cdots + a_{n-1} x_{n-1}) + \mu \varphi_M(x' - e_1'),
    $$
    where
    $$
    \varphi_M(x') := \min\{ |x'|^{-M}, \eta^{-M} \} - 1,
    $$
    and $ \mu $, $ \eta $ are chosen sufficiently small such that $ w < \gamma_{*} $ in $ B_\eta(e_1') $, and $ M $ is large enough such that graph of $ w $ has positive mean curvature in $ B_1(e_1') \setminus B_\eta(e_1') $.
    Thus, by a direct calculation, we know that there exists sufficient small constant $ \delta_0 > 0 $ such that
    \begin{equation}\label{superw}
         \Delta w-\frac{Dw^{T}D^2w Dw}{1+|Dw|^2} > \delta_0 \| D^2 w \|, \quad \| D^2 w \| < \delta_0^{-1} \quad \text{in} \quad B_1(e_1') \setminus B_\eta(e_1').
    \end{equation}
    Now, we claim that there exists a subsequence of $ \{ \varepsilon_k \} $, we may still denote it as $ \{ \varepsilon_k \} $, such that $ \{ u_{\varepsilon_k} = 0 \} $ continue to lie above the graph of $ w $, i.e. $\gamma_{\varepsilon_k}\geq w$ in $B_1(e_1')$.
    
    If this were not the case, by the boundary condition and uniform convergence, we know that there must exist a small domain in $ B_1(e_1') \setminus B_\eta(e_1') $ where $ \{ u_{\varepsilon_k} = 0 \} $ lies below the graph of $ w $. In this case, we can slide down the graph of $ w $ along the $x_n$-direction until it touches interior of $ \{ u_{\varepsilon_k} = 0 \} $.  By translating and rotating the graph of $w$ and $H_-$ we can assume in the new coordinates  $Dw$  in \eqref{superw} vanishes at the touching point.  Since the touching point lies in the interior, Proposition \ref{prop:Modica} leads to a contradiction. Therefore, there exists a subsequence $ \{ \varepsilon_k \} $ such that $ \{ u_{\varepsilon_k} = 0 \} $ still lies above the graph of $ w $. See Figure \ref{fig:1} for the geometric interpretation.

%%%%%%%
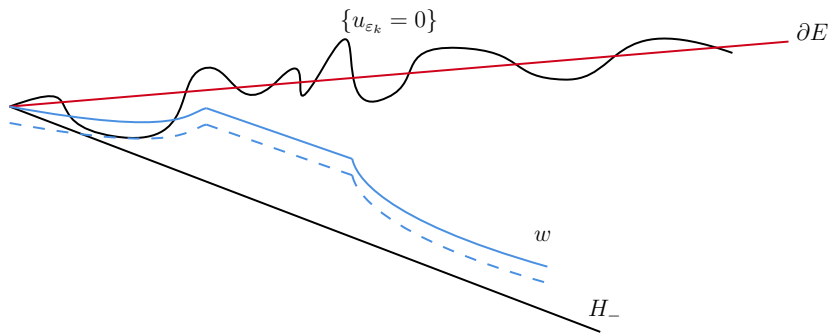
\begin{figure}[htbp]
        \centering
\tikzset{every picture/.style={line width=0.75pt}} %set default line width to 0.75pt        

\begin{tikzpicture}[x=0.75pt,y=0.75pt,yscale=-0.8,xscale=0.8]
%uncomment if require: \path (0,330); %set diagram left start at 0, and has height of 330

%Straight Lines [id:da9372818036725472] 
\draw [line width=0.75]    (102,126.83) -- (474.33,269) ;
%Curve Lines [id:da5009791314523879] 
\draw    (102,126.83) .. controls (157,107) and (115.67,137.67) .. (168.33,145.67) .. controls (221,153.67) and (205,109.67) .. (223.67,103) .. controls (242.33,96.33) and (244.33,137.67) .. (271.67,109.67) .. controls (299,81.67) and (271,151.67) .. (300.33,101) .. controls (329.67,50.33) and (304.33,132.33) .. (335,123) .. controls (365.67,113.67) and (341.67,88) .. (385,89.83) .. controls (428.33,91.67) and (415.67,109) .. (452,109.83) .. controls (488.33,110.67) and (478.33,65) .. (557.67,93) ;
%Straight Lines [id:da18522471949958685] 
\draw [color={rgb, 255:red, 208; green, 2; blue, 27 }  ,draw opacity=1 ]   (102,126.83) -- (593,85.83) ;
%Curve Lines [id:da39917985335007233] 
\draw [color={rgb, 255:red, 74; green, 144; blue, 226 }  ,draw opacity=1 ]   (102,126.83) .. controls (211,146.83) and (214,130.83) .. (226,127.83) ;
%Straight Lines [id:da4750462031697047] 
\draw [color={rgb, 255:red, 74; green, 144; blue, 226 }  ,draw opacity=1 ]   (226,127.83) -- (318,159.83) ;
%Curve Lines [id:da7203157828478801] 
\draw [color={rgb, 255:red, 74; green, 144; blue, 226 }  ,draw opacity=1 ]   (318,159.83) .. controls (320,169.83) and (335,197.83) .. (441,227.83) ;
%Curve Lines [id:da9956354329569377] 
\draw [color={rgb, 255:red, 74; green, 144; blue, 226 }  ,draw opacity=1 ] [dash pattern={on 4.5pt off 4.5pt}]  (102,137.17) .. controls (211,157.17) and (214,141.17) .. (226,138.17) ;
%Straight Lines [id:da21048263679950563] 
\draw [color={rgb, 255:red, 74; green, 144; blue, 226 }  ,draw opacity=1 ] [dash pattern={on 4.5pt off 4.5pt}]  (226,138.17) -- (318,170.17) ;
%Curve Lines [id:da6910472493252326] 
\draw [color={rgb, 255:red, 74; green, 144; blue, 226 }  ,draw opacity=1 ] [dash pattern={on 4.5pt off 4.5pt}]  (318,170.17) .. controls (320,180.17) and (335,208.17) .. (441,238.17) ;

% Text Node
\draw (309,64) node [anchor=north west][inner sep=0.75pt]  [xscale=0.75,yscale=0.75] [align=left] {$\displaystyle \{u_{\varepsilon _{k}} =0\}$};
% Text Node
\draw (596.67,73) node [anchor=north west][inner sep=0.75pt]  [xscale=0.75,yscale=0.75] [align=left] {$\displaystyle \partial E$};
% Text Node
\draw (465.33,247.33) node [anchor=north west][inner sep=0.75pt]  [xscale=0.75,yscale=0.75] [align=left] {$\displaystyle H_{-}$};
% Text Node
\draw (431.33,204) node [anchor=north west][inner sep=0.75pt]  [xscale=0.75,yscale=0.75] [align=left] {$\displaystyle w$};
\end{tikzpicture}
        \caption{Barrier function}
        \label{fig:1}
    \end{figure}
%%%%%%
    
    Following a similar argument to the proof of \cite[Lemma 3.1]{DMYZ}, we can complete the proof. Indeed, from the previous claim, we know that $ \gamma_{\varepsilon_k} \geq -a_n^{-1} (A_- x_1 + a_2 x_2 + \cdots + a_{n-1} x_{n-1}) + t x_1 $ on the line segment from $ 0' $ to $ e_1' $ for some small $ t > 0 $. We can repeat the same argument by replacing $ e_1' $ with $ e_1' + \lambda e' $ for any unit vector $ e' $ in the span of $ \{ e_2', \dots, e_{n-1}' \} $ and $ |\lambda| < \eta $. This allows us to conclude that $ \gamma_{\varepsilon_k} \geq -a_n^{-1} (A_- x_1 + a_2 x_2 + \cdots + a_{n-1} x_{n-1}) + t x_1 $ on the line segment from $ \lambda e' $ to $ e_1' + \lambda e' $ for all $ e' $ and $ \lambda $ as above. This implies the existence of a hyperplane between $ \{ u = 0 \} $ and $ H_- $, which forms a small positive angle with $ H_- $ on the boundary $ \Gamma $. This contradicts the definition of $ A_- $ in \eqref{eq:def-A}.  In the second case where $ A_- = +\infty $, we know that we can find a sufficiently small $ \mu > 0 $ such that the non-vertical hyperplane $ \widetilde{H} := \{ \mathbb{R}^n_+ : \mu^{-1} x_1 + a_2 x_2 + \cdots + a_n x_n = 0 \} $ lies between $ \partial E $ and $ H_- $, and is sufficiently close to $ H_- $. Then, we can replace $ H_- $ with $ \widetilde{H} $ in the above discussion and argue similarly as in the first case where  $ A_- <+\infty $ to derive a contradiction. Hence, we conclude that $ H_- = \partial E $. Similarly, we can show that $ H_+ = \partial E $. 
\end{proof}

Now, we can complete the proof of Theorem \ref{thm:main}.
%%%%%%
\begin{proof}[Proof of Theorem \ref{thm:main}]
    Combining Lemma \ref{lem:global-min}, Lemma \ref{lem:converges}, and Lemma \ref{lem:A_+=A_-}, we conclude that the level set $ \{ u = 0 \} $ coincides with $ H_{\pm} $, implying that $ \{ u = 0 \} $ is a hyperplane. Since we can apply this method to all level sets of the solution to \eqref{eq:half-Al-Ca}, we deduce that all level sets are flat. Moreover, since $ \partial_{x_n} u > 0 $, i.e., $ u $ is increasing with respect to $ x_n $, the level sets are parallel and we conclude that  $ u $ must be a one-dimensional solution. Due to its boundary condition, we have
    $$
    u(x) = g_0(\pm a_1 x_1 + a_2 x_2 + \cdots + a_n x_n).
    $$
    Then we complete the proof.
\end{proof}

%%%%%%%%%%%%%%%%%%%%%%%%%%%%%%%%%%%%%%%%%%%%%%%%%%%%%%%%%%%%%%%%%%%%%%%%%%%%%%%%%%%%%%%%%%%%%%%%%%%%%%%%%%%%
%%%%%%%%%%%%%%%%%%%%%%%%%%%%%%%%

\end{document}